\documentclass[reqno,12pt]{amsart}
\def\bege{\begin{equation}} \def\ende{\end{equation}}
\def\begr{\begin{eqnarray}} \def\endr{\end{eqnarray}}
\usepackage{amsmath}
\usepackage{amssymb}
\usepackage{cite}

\usepackage{bbm}
\usepackage{amsmath}
\usepackage{txfonts}
\usepackage{stmaryrd}
\usepackage{amssymb}
\usepackage{amsfonts}
\usepackage{times}
\usepackage{mathrsfs}

\usepackage{amsthm}
\usepackage{enumerate}

\usepackage{framed}
\usepackage{lipsum}
\usepackage{color}

\newcommand{\TT}{{\mathbb T}}

\newcommand{\DD}{{\mathbb D}}

\def\D{\mathbb{D}}
\def\N{\mathbb N}

\def\a{\alpha}

\def\t{\theta}

\def\begr{\begin{eqnarray}} \def\endr{\end{eqnarray}}

\def\msk{\medskip}
\def\ol{\overline}

\newtheorem{Lemma}{Lemma}[section]

\newtheorem{Proposition}[Lemma]{Proposition}

\newcounter{other}            

\begin{document}
\title[]{Closed range of generalized integration operators on analytic tent spaces}
	
\author{Rong Yang and Xiangling Zhu$\dagger$}
	 
\address{Rong Yang\\ Institute of Fundamental and Frontier Sciences, University of Electronic Science and Technology of China, 610054, Chengdu, Sichuan, P.R. China.}
\email{yangrong071428@163.com  }

\address{Xiangling Zhu
\\University of Electronic Science and Technology of China, Zhongshan Institute, 528402, Zhongshan, Guangdong, P.R.	China.}
\email{jyuzxl@163.com}


\begin{abstract} 
 
 In this paper, the necessary and sufficient conditions for the generalized integration operator \(T_g^{n,k}\) to have closed ranges on the analytic tent spaces are investigated. 
		
\thanks{$\dagger$ Corresponding author.}
\end{abstract}

\maketitle
	
\section{Introduction}

Let \(\mathbb{D}\) denote the open unit disk in the complex plane \(\mathbb{C}\), and let \(\mathbb{T}\) be its boundary. 
 For \(0 < \rho < 1\), define \(E(a,\rho)\) as the Euclidean disk centered at \(a\in\mathbb{D}\) with a radius of \(\rho(1 - |a|^{2})\). Specifically,
\[
E(a,\rho)=\left\{z\in\mathbb{D}:|z - a|<\rho(1 - |a|^{2})\right\}.
\]
For \(0 < r < 1\), let \(\Delta(a,r)\) denote  the pseudo-hyperbolic disk centered at \(a\in\mathbb{D}\) with a radius of \(r\). That is,
\[
\Delta(a,r)=\left\{z\in\mathbb{D}:\left|\frac{a - z}{1-\overline{a}z}\right|<r\right\}.
\]

Let \(H(\mathbb{D})\) denote  the class of all analytic functions on \(\mathbb{D}\). The Hardy space \( H^p  (0 < p < \infty) \), consists of all functions \( f \in H(\mathbb{D}) \) satisfying  
\[
\|f\|_{H^p}^p = \sup_{0 \leq r < 1} \int_0^{2\pi} |f(r e^{i\theta})|^p \frac{d\theta}{2\pi} < \infty.
\]
 For \(0 < p < \infty\) and \(\alpha > -1\), the weighted Bergman space \(A_\alpha^p\) is defined as the collection of all functions \(f \in H(\mathbb{D})\) such that
\[
\|f\|_{A_\alpha^p}^p = (\alpha + 1) \int_{\mathbb{D}} |f(z)|^p (1 - |z|^2)^\alpha dA(z) < \infty,
\]
where \(dA(z) = \frac{1}{\pi} \, dx \, dy\) denotes the normalized Lebesgue area measure on \(\mathbb{D}\). 

Let $0 < p,q<\infty$ and $\alpha>- 2$ and $\eta \in \mathbb{T}$. The tent space $T_p^q(\alpha)$ is the space consisting  of all measurable functions $f$ on $\mathbb{D}$ such that 
\begin{align}\label{2.111}
	\|f\|_{T_p^q(\alpha)}=\left(\int_\mathbb{T}\left(\int_{\Gamma(\eta)}|f(z)|^p(1 - |z|^2)^{\alpha} d A(z)\right)^{\frac{q}{p}}|d \eta|\right)^{\frac{1}{q}}<\infty,
\end{align}
where the non-tangential region $\Gamma(\eta)$ is defined as follows:
$$
\Gamma(\eta)=\Gamma_\zeta(\eta)=\left\{z \in \mathbb{D}:|z - \eta|<\zeta(1 - |z|^2), \zeta>\frac{1}{2}\right\}.
$$  
The aperture $\zeta$ of the non-tangential region $\Gamma_{\zeta}(\eta)$ is not explicitly denoted. This is due to the fact that for any two distinct apertures, the resulting function spaces are identical when equipped with equivalent quasinorms. 
We adopt the notation \(AT_p^q(\alpha)\) to denote the tent space of analytic functions, defined formally as \(A T_p^q(\alpha)=T_p^q(\alpha) \cap H(\mathbb{D})\). Notably, in the special case where \(q = p\), the space \(A T_p^p(\alpha)\) coincides with the weighted Bergman space \(A^p_{\alpha+1}\), establishing a fundamental connection to classical function spaces. 

Tent spaces, initially introduced by Coifman, Meyer, and Stein in \cite{cms}, serve as a pivotal tool in harmonic analysis. Their seminal framework provides a unified approach to investigating problems associated with canonical function spaces, including Hardy spaces and Bergman spaces. This construction has since become indispensable for establishing deep relationships between geometric structures and analytic properties in function theory.

Let $\mathbb{N}$ represent the set of positive integers. Consider $g \in H(\mathbb{D})$, $n \in \mathbb{N}$, and $k \in \mathbb{N} \cup \{0\}$ such that $0 \leq k < n$. The generalized integration operator $T_g^{n,k}$ is formally defined by  
\[
T_g^{n,k}f(z) := I^n \left( f^{(k)}(z) \cdot g^{(n - k)}(z) \right), \quad f \in H(\mathbb{D}),
\]  
where $I^n$ denotes the $n$-th iteration of the integration operator $I$, with $If(z) := \int_{0}^{z} f(t) \, dt$ specifying the base integration operation.   This operator $T_g^{n,k}$ was first introduced by Chalmoukis \cite{ch}. In the degenerate case $n = 1$, $k = 0$, it reduces to the classical form  
\[
T_g^{1,0}f(z) = \int_{0}^{z} f(w) g'(w) \, dw =: T_g f(z),
\]  
which recovers the Volterra integration operator initially studied by Pommerenke \cite{p} in 1977. Pommerenke established the fundamental result that $T_g$ is bounded on  $H^2$ if and only if $g$ belongs to the space $BMOA$ of bounded mean oscillation analytic functions. In 1995, Aleman and Siskakis \cite{as1} extended this characterization to the full range $p \geq 1$, proving that $T_g$ is bounded on $H^p$ precisely when $g \in BMOA$.

Chalmoukis \cite{ch} studied the boundedness of the generalized integration operator \(T_{g}^{n,k}\) on Hardy spaces. His results show that \(T_g^{n,k}\), as a generalization of \(T_g\), exhibits a distinctly different behavior. For example, he proved that \(T_g^{n,k}: H^p\rightarrow H^p\) is bounded if and only if \(g\in\mathcal{B}\) for \(k\geq1\). Recall that the   Bloch space \( \mathcal{B} \) is defined as the set of all functions \( f \in H(\mathbb{D}) \) for which
\[
\|f\|_{\mathcal{B}} = |f(0)| + \sup_{z \in \mathbb{D}} (1 - |z|^2) |f'(z)| < \infty.
\] 
 In \cite{dlq}, Du, Li and Qu characterized the boundedness, compactness, and Schatten class membership of the generalized integration operator \(T_g^{n,k}\) on Bergman spaces with doubling weights. This line of study highlights the deep connection between operator theory and function spaces in complex analysis.  For further studies on the operators \(T_g\) and \(T_g^{n,k}\), see \cite{dlq,mppw,hyl,lly,llw,zz,w2} and the references therein.

Recently, the investigation into the closed  range of the operator \( T_g \) acting on diverse function spaces has become an area of interest. This study particularly focuses on characterizing the structural properties and analytical conditions under which the range of \( T_g \) maintains closedness.  For one-to-one operators on quasi-Banach space $(X,\|\cdot\|)$, the property of having closed range is equivalent to being bounded below.  Recall that a linear operator $T$ on a quasi-Banach space $(X,\|\cdot\|)$ is said to be bounded below if there exists $C>0$ such that $$\|T y\| \geq C\|y\|$$ for all $y \in X$.  In \cite{an}, Anderson showed that the integration operator \(T_g\) never has  closed range on \(H^2\), \(\mathcal{B}\), or \(BMOA\).
However, it can have closed range on weighted Bergman spaces. Chen  \cite{chen24}   proved that for the function space $F(p,p\alpha - 2,s)$ with the conditions $0<\alpha\leq1$, $s>0$, and $p>0$ such that $s + p\alpha>1$, the operator $T_g$ also fails to have a closed range. 
 See \cite{ag2,akfu,ham2020, m2023, caoh,liangliu, m2025,et,liuli2016} for more results about the closed range of integral operators and composition operators. 
 
In this paper, we   study   the closed range of the operator \(T_g^{n,k}\) on the analytic tent spaces \(AT_p^q(\alpha)\). The main result can be   stated as follows.

{\bf Main Theorem} \label{am} 
{\it Let $1 \leq p, q<\infty,\a>-2$ and $g \in \mathcal{B}$, $n\in\N$ and $k\in\N\cup\{0\}$ such that $0\le  k< n$.
For $c>0$, let $$G_c^{n,k}=\left\{z \in \mathbb{D}:|g^{(n-k)}(z)|(1-|z|^2)^{n-k}>c\right\}.$$
The following statements are equivalent:

\begin{enumerate}
\item [(i)]  $T_g^{n,k}: AT_p^q(\a)/\Upsilon_k \rightarrow AT_p^q(\a)/\Upsilon_k$ has closed range.

\item [(ii)] For all \( a \in \mathbb{D} \), there exist \( \delta > 0 \) and \( 0 < \rho < 1 \) such that
\[
A(G_c^{n,k} \cap \Delta(a, \rho)) \geq \delta A(\Delta(a, \rho)).
\]

\item [(iii)] For all \( a \in \mathbb{D} \), there exist \( \delta > 0 \) and \( 0 < \rho < 1 \) such that
\[
A(G_c^{n,k} \cap E(a, \rho)) \geq \delta A(E(a, \rho)).
\]
\end{enumerate}
Here $\Upsilon_0=\emptyset$, $ \Upsilon_k=\bigcup_{i = 0}^{k - 1}P_i, ~~~~ k=1,2,3,\cdot\cdot\cdot, $  \(P_i\) is the set of all polynomials with degree \(i\). }

 The paper is organized as follows. In Section 2, several lemmas are introduced. In Section 3, we give a detail proof for the   main result  of this paper.

Throughout this paper,    we write \(E\lesssim F\) to denote the existence of a positive constant \(C\) such that \(E\leq CF\). Furthermore, \(E\asymp F\) indicates that both \(E\lesssim F\) and \(F\lesssim E\) are satisfied.

\section{Preliminaries} 
In this section, we state some notations and lemmas, which will be used in the proof of main theorem.

\begin{Lemma}\cite[Lemma 2.5]{of}\label{2.5}
If $s>-1$,$r,t>0$, $r+t-s>2$ and  $t>s+2>r$, then 
$$
\int_{\D}\frac{(1-|w|^2)^s}{|1-\ol{b}w|^r|1-\ol{z}w|^t}dA(w)\lesssim\frac{1}{|1-\ol{b}z|^r(1-|z|^2)^{t-s-2}}
$$
for all $b,z\in\D$.
\end{Lemma}

\begin{Lemma}\cite[Lemma 4.1]{mj} \label{inn}
Let  $0<p<\infty$ and $b\in\D$. Then, as $|b| \rightarrow 1$,
\begin{equation*}
\int_{0}^{2\pi}\frac{d\t}{|1-\ol{b}e^{i\t}|^p}
\asymp\left\{ 
\begin{aligned}
&1 &\text {if} ~& ~~0<p<1,\\
&\log\frac{1}{(1-|b|^2)}, &\text {if} ~&~~ p=1, \\
&\frac{1}{(1-|b|^2)^{p-1}} &\text {if} ~&~~ p>1.
\end{aligned}
\right.
\end{equation*}
The constants involved just depend on the parameter $p$.
\end{Lemma}

For $0<\tau<1$,  the cone-like region is defined as
$$
\Gamma_{\tau}(\eta)=\{z\in\D:|z|<\tau\}\cup \bigcup_{|z|<\tau}[z,\eta).
$$
It is of great importance that in $(\ref{2.111})$ we can use any of the cone-like regions. 
In other words, we have 
$$
\|f\|_{T_p^q(\alpha)}\asymp\left(\int_\mathbb{T}\left(\int_{\Gamma_{\tau}(\eta)}|f(z)|^p(1-|z|^2)^{\a} d A(z)\right)^{\frac{q}{p}}|d \eta|\right)^{\frac{1}{q}}.
$$
This is valid due to the following technical lemma.

\begin{Lemma}\cite[Lemma 4]{ar}\label{2.3}
Let $0<p, q<\infty$ and $\lambda>\max \left\{1, \frac{p}{q}\right\}$. Then there are constants $M_1=M_1(p, q, \lambda, K)$ and $M_2=M_2(p, q, \lambda, K)$ such that
$$
M_1 \int_{\mathbb{T}} \mu(\Gamma_K(\eta))^{\frac{q}{p}}|d \eta| \leq \int_{\mathbb{T}}\left(\int_{\mathbb{D}}\left(\frac{1-|z|^2}{|1-z \ol{\eta}|}\right)^\lambda d \mu(z)\right)^{\frac{q}{p}}|d \eta| \leq M_2 \int_{\mathbb{T}} \mu(\Gamma_K(\eta))^{\frac{q}{p}}|d \eta|
$$
for every positive measure $\mu$ on $\mathbb{D}$. Here $\Gamma_K(\eta)$ stands for any of the $\Gamma_\tau(\eta)$ or $\Gamma_\zeta(\eta)$. 
\end{Lemma}

\begin{Lemma}\label{Tfn}\cite[Theorem 2]{pa}
	Let $0<p,q<\infty$, $\a>-2$ and $n\in\N\cup\{0\}$. Then $f\in AT_p^q(\a)$ if and only if
	$$
	\int_{\mathbb{T}}\left(\int_{\Gamma{(\eta)}}|f^{(n)}(z)|^{p}(1-|z|^2)^{np+\a} dA(z)\right)^{\frac{q}{p}}|d\eta|<\infty.
	$$
\end{Lemma}

According to Theorem 3.1 in  \cite{yhl}, we have the following lemma.
\begin{Lemma}
	Let $0< p, q<\infty,\a>-2$, $g \in \mathcal{B}$, $n\in\N$ and $k\in\N\cup\{0\}$ such that $0\le  k< n$. Then $T_g^{n,k}: AT_p^q(\a)\to AT_p^q(\a)$ is bounded if and only if $g\in\mathcal{B}$.
\end{Lemma}

\begin{Proposition}\label{3.5}
Let \(0 < p, q < \infty\) and \(\alpha > -2\). For any given \(\varepsilon > 0\), there exists \(0 < \rho < 1\) such that for all \(a \in \mathbb{D}\), there is a function \(f_a \in A T_p^q(\alpha)\)    such that the following two statements hold true.
\begin{enumerate}
\item [(i)]  There exist two positive constants $C'$ and $C''$ satisfying
   $$
C''\leq \left\|f_a\right\|_{AT_p^q(\alpha)}   \leq C'. $$

\item [(ii)]  $$\left\|f_a \chi_{\Delta^c(a, \rho)}\right\|_{AT_p^q(\alpha)} \leq \varepsilon.
$$
\end{enumerate}
Here $\Delta^c(a, \rho)= \DD \setminus \Delta(a, \rho) $.
\end{Proposition}

\begin{proof}  $(i)$ Fix \(\lambda > \max \left\{1, \frac{p}{q}\right\}\). Set
\[
f_a(z) = \frac{(1 - |a|^2)^{\gamma- \frac{\alpha + 2}{p} - \frac{1}{q}}}{(1 - \ol{a} z)^\gamma}, \quad z \in \mathbb{D},
\]
where \(\gamma >\max \left\{  \frac{\alpha + 2}{p} +\frac{1}{q}, \frac{\lambda + 2\alpha + 4}{p}  \right\}\). 
Notice that for $z\in\Gamma(\eta)$, $|1-\ol{\eta}z|\asymp 1-|z|^2$.
Using Lemmas  \ref{2.5} and \ref{inn}, there exists a positive constant $C'$ such  that 
\begin{align*}
&\int_{\mathbb{T}}\left(\int_{\Gamma(\eta)} |f_a(z)|^p (1-|z|^2)^{\a}dA(z)\right)^{\frac{q}{p}}|d \eta|\\
=&(1-|a|^2)^{q\gamma-\frac{q(\a+2)}{p}-1}\int_{\mathbb{T}}\left(\int_{\Gamma(\eta)}  \frac{(1-|z|^2)^{\a}}{|1-\ol{a}z|^{p\gamma}}dA(z)\right)^{\frac{q}{p}}|d \eta|\\
\lesssim&(1-|a|^2)^{q\gamma-\frac{q(\a+2)}{p}-1}\int_{\mathbb{T}}\left(\int_{\D}  \frac{(1-|z|^2)^{\a+\lambda}}{|1-\ol{a}z|^{p\gamma}|1-\ol{\eta} z|^\lambda}dA(z)\right)^{\frac{q}{p}}|d \eta|\\
\lesssim&(1-|a|^2)^{q\gamma-\frac{q(\a+2)}{p}-1}\int_{\mathbb{T}}\left(\frac{1}{(1-|a|^2)^{p\gamma-(\a+2)-\lambda}|1-\ol{a}\eta|^\lambda}
\right)^{\frac{q}{p}}|d \eta|\\
\lesssim&(1-|a|^2)^{\frac{\lambda q}{p}-1}\int_{\TT}\frac{|d\eta|}{|1-\ol{a}\eta|^{\frac{\lambda q}{p}}}  \\
\leq& C'.
\end{align*}

On the other hand, without loss of generality, we may assume that $\frac{1}{2}\leq a <1$.
For $\eta=e^{i\theta}$, $z=re^{i\phi}\in\Gamma(\eta)$ satisfying
\[
|r e^{i\phi} - e^{i\theta}| < \zeta(1 - r^2).
\]
When \( r\to 1^{-} \) , this condition can be approximated as:  
$
|\phi - \theta| < C(1 - r),
$ 
where \( C \) is a constant. This shows that  
\begin{align*}
 &\int_{\Gamma(\eta)} \frac{(1 - r^2)^\alpha}{|1 - az|^{p\gamma}} \, dA(z) \\ 
=& \int_{\Gamma(e^{i\theta})} \frac{(1 - r^2)^\alpha}{|1 - ar e^{i\phi}|^{p\gamma}}  \frac{r}{\pi}  dr d\phi \\
\gtrsim&
\int_{a}^{1} \frac{(1 - r)^{\alpha+1}}{|1 - a r e^{i\theta}|^{p\gamma}}  dr.
\end{align*}
Here we used the length of the angular integration interval is equivalent to \((1 - r)\).
It is easy to check that $|1-are^{i\theta}|\leq \sqrt{2}(1-ar)$ when $0<\theta<1-a$ and $0\leq r<1$.
Using the fact that $1 - ar \asymp 1 - a$ for  $a < r < 1$,
there exists a positive constant $C''$ such  that
\begin{align*}
&\int_{\mathbb{T}}\left(\int_{\Gamma(\eta)} |f_a(z)|^p (1-|z|^2)^{\a}dA(z)\right)^{\frac{q}{p}}|d \eta|\\
=&(1-a^2)^{q\gamma-\frac{q(\a+2)}{p}-1}\int_{\mathbb{T}}\left(\int_{\Gamma(\eta)}  \frac{(1-|z|^2)^{\a}}{|1-az|^{p\gamma}}dA(z)\right)^{\frac{q}{p}}|d \eta|\\
\gtrsim&(1-a)^{q\gamma-\frac{q(\a+2)}{p}-1} 
\int_{0}^{2\pi} \left( \int_{a}^{1} \frac{(1-r)^{\alpha+1}}{|1-are^{i\theta}|^{p\gamma}} dr \right)^{\frac{q}{p}}  d\theta\\
\gtrsim&(1-a)^{q\gamma-\frac{q(\a+2)}{p}-1} 
\int_{0}^{1-a} \left( \int_{a}^{1} \frac{(1-r)^{\alpha+1}}{(1-ar)^{p\gamma}} dr \right)^{\frac{q}{p}}  d\theta\\
\asymp& (1-a)^{q\gamma-\frac{q(\a+2)}{p}-1} 
\int_{0}^{1-a} \left( \int_{a}^{1} \frac{(1-r)^{\alpha+1}}{(1-a)^{p\gamma}} dr \right)^{\frac{q}{p}}  d\theta\\
\asymp& (1-a)^{q\gamma-\frac{q(\a+2)}{p}-1} 
\int_{0}^{1-a}\frac{1}{(1-a)^{q\gamma-\frac{q(\alpha+2)}{p}}}d\theta\\
  \geq&  C''.
\end{align*}

$(ii)$ We will show that for any given \(\varepsilon > 0\), there exists \(0 < \rho < 1\) such that
$$
J=\left(\int_{\mathbb{T}}\left(\int_{\Gamma(\eta)} |f_a(z)|^p\chi_{\Delta^c(a, \rho)}(z) (1-|z|^2)^{\a}dA(z)\right)^{\frac{q}{p}}|d \eta|\right)^\frac{1}{q}  \leq \varepsilon.
$$
Note that for $z\in\Gamma(\eta)$, $|1-\ol{\eta}z|\asymp 1-|z|^2$. 
Thus,
$$
J\lesssim\left( \int_{\mathbb{T}}\left(\int_{\mathbb{D}}\left(\frac{1-|z|^2}{|1-\ol{\eta} z|}\right)^\lambda \left|f_a(z)\right|^p\chi_{\Delta^c(a, \rho)}(z)(1-|z|^2)^{\a}dA(z)\right)^{\frac{q}{p}}|d \eta| \right)^\frac{1}{q}.
$$
Using the change of variable $z \mapsto \varphi_a(z)$,
where $\varphi_a(z)=\frac{a-z}{1-\ol{a} z}$, we have
$$
\begin{aligned}
J & \lesssim \left(
\int_{\mathbb{T}}\left(\int_{\mathbb{D} \backslash \Delta(0, \rho)} \frac{(1-\left|\varphi_a(z)\right|^2)^{\lambda+\a}}{\left|1-\ol{\eta} \varphi_a(z)\right|^\lambda}\left|f_a\left(\varphi_a(z)\right)\right|^p|\varphi_a^{\prime}(z)|^2 d A(z)\right)^{\frac{q}{p}}|d \eta| \right)^\frac{1}{q}\\
& \lesssim \left(
\int_{\mathbb{T}}\left(\int_{\mathbb{D} \backslash \Delta(0, \rho)} \frac{(1-\left|z\right|^2)^{\lambda+\a}}{\left|1-\ol{\eta} \varphi_a(z)\right|^\lambda}\left|f_a\left(\varphi_a(z)\right)\right|^p|\varphi_a^{\prime}(z)|^{\lambda+\a+2} d A(z)\right)^{\frac{q}{p}}|d \eta| \right)^\frac{1}{q}\\
& \lesssim \left(
\int_{\mathbb{T}}\left(\int_{\mathbb{D} \backslash \Delta(0, \rho)} \frac{(1-\left|z\right|^2)^{\lambda+\a} |1-\ol{a}z|^\lambda}{\left|1-\overline{\varphi_a(\eta)} z\right|^\lambda |1-\ol{a}\eta|^\lambda}\left|f_a\left(\varphi_a(z)\right)\right|^p|\varphi_a^{\prime}(z)|^{\lambda+\a+2} d A(z)\right)^{\frac{q}{p}}|d \eta| \right)^\frac{1}{q}\\
& \lesssim \left(
\int_{\mathbb{T}}\left(\int_{\mathbb{D} \backslash \Delta(0, \rho)} \frac{(1-\left|z\right|^2)^{\lambda+\a} }{\left|1-\overline{\varphi_a(\eta)} z\right|^\lambda |1-\ol{a}\eta|^\lambda}(1-|a|^2)^{\lambda-\frac{p}{q}}|1-\ol{a}z|^{p\gamma-\lambda-2\a-4} d A(z)\right)^{\frac{q}{p}}|d \eta| \right)^\frac{1}{q}
\end{aligned}
$$
$$
\begin{aligned}
&\lesssim\left(\int_{\mathbb{T}} \frac{(1-|a|^2)^{\frac{\lambda q}{p}-1}}{|1-\ol{a} \eta|^{\frac{\lambda q}{p}}}\left(\int_{\mathbb{D} \backslash \Delta(0, \rho)} \frac{(1-|z|^2)^{\lambda+\a}|1-\ol{a} z|^{p\gamma-\lambda-2\a-4}}{\left|1-\overline{\varphi_a(\eta)} z\right|^\lambda} d A(z)\right)^{\frac{q}{p}}|d \eta|\right)^\frac{1}{q}. 
\end{aligned}
$$
Since \(\gamma >  \frac{\lambda + 2\alpha + 4}{p}\), we have
\(
|1 - \ol{a} z|^{p\gamma  - \lambda - 2\alpha - 4} \leq 2^{p\gamma  - \lambda - 2\alpha - 4}.
\)
Hence,
$$
\begin{aligned}
J & \lesssim\left( \int_{\mathbb{T}} \frac{(1-|a|^2)^{ \frac{\lambda q}{p}-1}}{|1-\ol{a} \eta|^{\frac{\lambda q}{p}}}\left(\int_{\mathbb{D} \backslash \Delta(0, \rho)} \frac{(1-|z|^2)^{\lambda+\a}}{\left|1-\overline{\varphi_a(\eta)} z\right|^\lambda} d A(z)\right)^{\frac{q}{p}}|d \eta| \right)^\frac{1}{q}\\
& =\left(\int_{\mathbb{T}} \frac{(1-|a|^2)^{ \frac{\lambda q}{p}-1}}{|1-\ol{a} \eta|^{\frac{\lambda q}{p}}}\left(\int_\rho^1 \frac{1}{\pi} \int_0^{2 \pi} \frac{d \theta}{|1-\overline{\varphi_a(\eta)} r e^{i \theta}|^\lambda} (1-r^2)^{\lambda+\a} r d r\right)^{\frac{q}{p}}|d \eta| \right)^\frac{1}{q}.
\end{aligned}
$$
Using the fact that $\left|\varphi_a(\eta)\right|=1=|\eta|$, $\lambda>\max \left\{1, \frac{p}{q}\right\}$ and Lemma \ref{inn}, we get
$$
J\lesssim \left((1-|a|^2)^{ \frac{\lambda q}{p}-1}(1-\rho)^{\frac{q(\a+2)}{p}}\int_{\mathbb{T}} \frac{|d \eta|}{|1-\ol{a} \eta|^{ \frac{\lambda q}{p}}} \right)^\frac{1}{q}.
$$
Applying Lemma \ref{inn} once more, we obtain
\[
J \lesssim (1 - \rho)^{ \frac{\alpha + 2}{p}}.
\]
Hence, for any given \(\varepsilon > 0\), there exists \(0 < \rho < 1\) such that $\left\|f_a \chi_{\Delta^c(a, \rho)}\right\|_{AT_p^q(\a)} \leq \varepsilon$.  The proof is complete.   
\end{proof}

\begin{Lemma}\cite{l0}\label{111}
	If $0<\rho<1$ and $\frac{2 \rho}{1+\rho^2} \leq r<1$, then
	$$
	E(a,\rho) \subseteq \Delta(a, r),\quad a\in\D.
	$$
\end{Lemma}

The following can be found in \cite{ag2} or \cite{pk}.

\begin{Lemma}\label{222}
	Let $0<\rho,\tau<1$ and $\eta\in\mathbb{T}$. Then there exists  $ {\tau'}\in(\tau,1)$ such that
	$$
	\cup_{a \in \Gamma_\tau(\eta)} E(a,\rho) \subseteq \Gamma_{{\tau'}}(\eta).
	$$ 
\end{Lemma}

The following two lemmas play an important role in the proof of the main theorem. 

\begin{Lemma}\label{a}
Let \( 1 \leq p < \infty \), \( \alpha > -2 \), \( \varepsilon > 0 \), \( 0 <   \rho, \tau< 1 \), $\eta\in\mathbb{T}$, $n\in\mathbb{N}\cup\{0\}$ and \( f \in H(\mathbb{D}) \). Then,   there exists  $ {\tau'}\in(\tau,1)$ and a positive constant \( C_1  \)  depending only on \( \rho \), such that
\[
\int_{\Omega \cap \Gamma_\tau(\eta)} |f^{(n)}(z)|^p (1 - |z|^2)^{np + \alpha} \, dA(z) \leq C_1 \varepsilon \int_{\Gamma_{\tau'}(\eta)} |f^{(n)}(z)|^p (1 - |z|^2)^{np + \alpha} \, dA(z),
\]
where  \[
\Omega= \left\{ a \in \mathbb{D} : |f^{(n)}(a)|^p  < \frac{\varepsilon}{A(E(a, \rho))} \int_{E(a, \rho)} |f^{(n)}(z)|^p  \, dA(z) \right\}.
\]
\end{Lemma} 

\begin{proof} If $a \in \Omega$, then
$$
|f^{(n)}(a)|^p<\frac{\varepsilon}{A\left(E(a,\rho)\right)} \int_{E(a,\rho)}|f^{(n)}(z)|^p d A(z) .
$$
Using   Lemma \ref{222} and Fubini's theorem,   there exists $ {\tau'}\in(\tau,1)$ such that 
$$
\begin{aligned}
&\int_{\Omega \cap \Gamma_\tau(\eta)}|f^{(n)}(a)|^p(1-|a|^2)^{np+\a} d A(a) \\
<&\varepsilon \int_{\Omega \cap \Gamma_\tau(\eta)}\left(\frac{1}{A\left(E(a,\rho)\right)} \int_{E(a,\rho)}|f^{(n)}(z)|^p d A(z)\right) (1-|a|^2)^{np+\a}d A(a) \\
\leq& \varepsilon \int_{\Omega \cap \Gamma_\tau(\eta)}\left(\frac{1}{A\left(E(a,\rho)\right)} \int_{\Gamma_{\tau'}(\eta)}|f^{(n)}(z)|^p\chi_{E(a,\rho)}(z)  d A(z)\right) (1-|a|^2)^{np+\a}d A(a) \\
\le&\varepsilon \int_{\Gamma_{\tau'}(\eta)}|f^{(n)}(z)|^p \int_{\Omega \cap \Gamma_\tau(\eta)}\frac{\chi_{\Delta(z, r)}(a)}{A\left(E(a,\rho)\right)}(1-|a|^2)^{np+\a}d A(a)d A(z) 
\end{aligned}
$$
when $\frac{2 \rho}{1+\rho^2} \leq r<1$.  The last inequality used the fact that 
$$\chi_{E(a,\rho)}(z) \leq \chi_{\Delta(a, r)}(z)=\chi_{\Delta(z,r)}(a)$$ 
when $\frac{2 \rho}{1+\rho^2} \leq r<1$. 
In addition, $A\left(E(a,\rho)\right) \asymp\rho^2(1-|a|^2)^2$. Hence, 
\begin{align*}
 \int_{\Omega \cap \Gamma_\tau(\eta)}\frac{\chi_{\Delta(z, r)}(a)}{A\left(E(a,\rho)\right)}(1-|a|^2)^{np+\a} d A(a) 
\lesssim&  \int_{\Delta(z,r)}\frac{(1-|a|^2)^{np+\a}}{\rho^2(1-|a|^2)^2}dA(a)\\
\lesssim&\frac{(1-|z|^2)^{np+\a}}{\rho^2(1-|z|^2)^{2}}A(\Delta(z,r)) \le C_1(1-|z|^2)^{np+\a},
\end{align*}
where the constant $C_1$ depends only on $\rho$. Therefore,
$$
\int_{\Omega \cap \Gamma_\tau(\eta)}|f^{(n)}(z)|^p(1-|z|^2)^{np+\a} A(z) \leq C_1 \varepsilon \int_{\Gamma_{\tau'}(\eta)}|f^{(n)}(z)|^p(1-|z|^2)^{np+\a} dA(z) .
$$
The proof is complete. 
\end{proof}

Let $1 \leq p<\infty$, $0<\rho, \lambda<1$, $ a \in \mathbb{D}$ and $n\in\mathbb{N}\cup\{0\}$. Define 
$$
E_\lambda(a)=\left\{z \in E(a,\rho):|f^{(n)}(z)|^p \geq \lambda|f^{(n)}(a)|^p\right\}, 
$$
and 
$$
B_\lambda f(a)=\frac{1}{A\left(E_\lambda(a)\right)} \int_{E_\lambda(a)}|f^{(n)}(z)|^p d A(z).
$$

\begin{Lemma}\label{b}
Let \( 1 \leq p < \infty \), \( \alpha > -2 \), \( 0 < \varepsilon,\rho,\tau < 1 \), \( 0 < \lambda < \frac{1}{2^p} \), $\eta\in\mathbb{T}$, $n\in\mathbb{N}\cup\{0\}$ and \( f \in H(\mathbb{D}) \). Then,  there exists   $ {\tau'}\in(\tau,1)$ and a positive constant \( C_2  \) depending only on \( \rho \),  such that
\[
\int_{B \cap \Gamma_\tau(\eta)} |f^{(n)}(z)|^p (1 - |z|^2)^{np + \alpha} \, dA(z) \leq C_2 \varepsilon \int_{\Gamma_{\tau'}(\eta)} |f^{(n)}(z)|^p (1 - |z|^2)^{np + \alpha}  dA(z),
\]
where \[
B = \left\{ a \in \mathbb{D} : |f^{(n)}(a)|^p < \varepsilon^{1 + \frac{2}{p}} B_\lambda f(a) \right\}.
\]
\end{Lemma}

\begin{proof} If $a \in B$, then
$$
|f^{(n)}(a)|^p<\varepsilon^{1+\frac{2}{p}} B_\lambda f(a)
=\frac{\varepsilon^{1+\frac{2}{p}}}{A\left(E_\lambda(a)\right)} \int_{E_\lambda(a)}|f^{(n)}(z)|^pd A(z).
$$
By Lemma \ref{a}, there exists  $ {\tau'}\in(\tau,1)$ and a positive constant \( C_1  \) depending only on \( \rho \),  such that 
$$
\begin{aligned}
&\int_{B \cap \Gamma_\tau(\eta)}|f^{(n)}(a)|^p(1-|a|^2)^{np+\a}d A(a)\\
=&\int_{B \cap \Gamma_\tau(\eta)\cap \Omega}|f^{(n)}(a)|^p(1-|a|^2)^{np+\a}d A(a)+\int_{\left(B \cap \Gamma_\tau(\eta)\right) \backslash \Omega}|f^{(n)}(a)|^p(1-|a|^2)^{np+\a}d A(a)\\
\le&C_1 \varepsilon \int_{\Gamma_{\tau'}(\eta)}|f^{(n)}(a)|^p(1-|a|^2)^{np+\a} dA(a)+\int_{\left(B \cap \Gamma_\tau(\eta)\right) \backslash \Omega}|f^{(n)}(a)|^p(1-|a|^2)^{np+\a}d A(a).
\end{aligned}
$$
Using   Lemma \ref{222} and Fubini's theorem,  we get
\begin{align*}
&\int_{\left(B \cap \Gamma_\tau(\eta) \right) \backslash \Omega}|f^{(n)}(a)|^p(1-|a|^2)^{np+\a}dA(a)\\
\le&\varepsilon^{1+\frac{2}{p}} \int_{\left(B \cap \Gamma_\tau(\eta)\right) \backslash \Omega} \frac{1}{A\left(E_\lambda(a)\right)} \int_{E_\lambda(a)}|f^{(n)}(z)|^p d A(z) (1-|a|^2)^{np+\a}dA(a)\\
\le&\varepsilon^{1+\frac{2}{p}} \int_{\left(B \cap \Gamma_\tau(\eta)\right) \backslash \Omega} \frac{1}{A\left(E_\lambda(a)\right)} \int_{\Gamma_{\tau'}(\eta)}|f^{(n)}(z)|^p \chi_{E_\lambda(a)}(z) d A(z) (1-|a|^2)^{np+\a}dA(a)\\
\le&\varepsilon^{1+\frac{2}{p}} \int_{\Gamma_{\tau'}(\eta)}|f^{(n)}(z)|^p\int_{\left(B \cap \Gamma_\tau(\eta)\right) \backslash \Omega}\frac{\chi_{E_\lambda(a)}(z)}{A\left(E_\lambda(a)\right)} (1-|a|^2)^{np+\a}dA(a)dA(z)\\
\le&\varepsilon^{1+\frac{2}{p}} \int_{\Gamma_{\tau'}(\eta)}|f^{(n)}(z)|^p\int_{\left(B \cap \Gamma_\tau(\eta)\right) \backslash \Omega}\frac{\chi_{E(a,\rho)}(z)}{A\left(E_\lambda(a)\right)} (1-|a|^2)^{np+\a}dA(a)dA(z),
\end{align*}
where the last inequality is justified by the inclusion $E_\lambda(a)\subset E(a,\rho)$.

 Assume that  
\begin{align}\label{2.2}
 \left\{z \in \mathbb{D}:|z-a|<\frac{\varepsilon^{\frac{1}{p}} \rho(1-|a|^2)}{2 C}\right\} \subset E_\lambda(a), \quad \forall a \notin \Omega,
 \end{align}
 where $C$ is a positive constant to be determined later. Then
$$
A(E_\lambda(a))\ge\frac{\varepsilon^{\frac{2}{p}}}{4C^2}A(E(a,\rho)),\quad \forall a \notin \Omega.
$$
Using (\ref{2.2}) and the fact that  $\chi_{E(a,\rho)}(z) \leq \chi_{\Delta(a, r)}(z)=\chi_{\Delta(z,r)}(a)$ 
when $\frac{2 \rho}{1+\rho^2} \leq r<1$, we get
$$
\begin{aligned}
&\int_{\left(B \cap \Gamma_\tau(\eta) \right) \backslash \Omega}|f^{(n)}(a)|^p(1-|a|^2)^{np+\a}dA(a)\\
\le&4C^2\varepsilon \int_{\Gamma_{\tau'}(\eta)}|f^{(n)}(z)|^p\int_{\left(B \cap \Gamma_\tau(\eta)\right) \backslash \Omega}\frac{\chi_{E(a,\rho)}(z)}{A\left(E(a,\rho)\right)} (1-|a|^2)^{np+\a}dA(a)dA(z)\\
\le& 4C^2\varepsilon \int_{\Gamma_{\tau'}(\eta)}|f^{(n)}(z)|^p\int_{\left(B \cap \Gamma_\tau(\eta)\right) \backslash \Omega}\frac{\chi_{\Delta(z,r)}(a)}{A\left(E(a,\rho)\right)} (1-|a|^2)^{np+\a}dA(a)dA(z)\\
\le& C\varepsilon\int_{\Gamma_{\tau'}(\eta)}|f^{(n)}(z)|^p(1-|z|^2)^{np+\a}dA(z).
\end{aligned}
$$

To complete the proof,   we only need to prove the formula (\ref{2.2}).
Let $a\in \D$ and $z\in E(a,\rho)$. Using Cauchy integral formula, we get
\begin{align*}
&|f^{(n)}(z)-f^{(n)}(a)|  
=  \frac{1}{2\pi} \left| \int_{|w - a| = \frac{\rho}{2}(1 - |a|)} f^{(n)}(w) \left( \frac{1}{w - z} - \frac{1}{w - a} \right) dw \right| \\
\leq& \frac{1}{2\pi} \int_{|w - a| = \frac{\rho}{2}(1 - |a|)} |f^{(n)}(w)| \left| \frac{z - a}{(w - z)(w - a)} \right| |dw|. 
\end{align*}
By subharmonicity, for \( w \in \overline{E(a,\frac{\rho}{2})} \), we obtain
\[
|f^{(n)}(w)| \leq \frac{C_3}{A(E(a,\rho))} \int_{E(a,\rho)} |f^{(n)}(u)| dA(u),
\]
where \( C_3 >0\). Therefore,
\[
|f^{(n)}(z)-f^{(n)}(a)| \leq \frac{4C_3 }{\rho(1 - |a|)} \frac{|z - a|}{A(E(a,\rho))} \int_{E(a,\rho)} |f^{(n)}(u)| dA(u).
\]
Now, if we consider \( z \in E(a,\frac{\varepsilon^{\frac{1}{p}}\rho}{2C}) \) for a \( C \) large enough, then
\[
|f^{(n)}(z)-f^{(n)}(a)| \leq \frac{\varepsilon^{\frac{1}{p}}}{2A(E(a,\rho))} \int_{E(a,\rho)} |f^{(n)}(u)| dA(u).
\]
Hence, if \( a \notin \Omega \), we get
\[
|f^{(n)}(z)| \geq |f^{(n)}(a)| - |f^{(n)}(a) - f^{(n)}(z)| \geq \frac{1}{2} |f^{(n)}(a)|.
\]
Therefore,
\[
|f^{(n)}(z)|^p \geq \frac{|f^{(n)}(a)|^p}{2^p} > \lambda |f^{(n)}(a)|^p,
\]
and this implies (\ref{2.2}). The proof is complete.
\end{proof}

As the proof of the following lemma employs an approach analogous to that of Lemma 1 in \cite{l0}, the proof is hereby omitted.

\begin{Lemma}\label{1.}
Let $1\le p<\infty$, $0<\lambda,\rho<1$, $n\in\mathbb{N}\cup\{0\}$ and $f\in H(\D)$. Then 
\begin{align*}
\frac{A\left(E_\lambda(a)\right)}{A\left(E(a,\rho)\right)} \geq \frac{\log \left(\frac{1}{\lambda}\right)}{\log \left(\frac{B_\lambda f(a)}{|f^{(n)}(a)|^p}\right)+\log \left(\frac{1}{\lambda}\right)},  \quad a\in\D.
\end{align*}
\end{Lemma}\msk  \vskip 4mm

\section{Proof of Main Theorem  }

\begin{proof}

$(i)\Rightarrow(ii)$  For $f \in AT_p^q(\a)/\Upsilon_k$, assume there exists a constant $C_4>0$ such that 
$$
\| T_g^{n,k}f \|_{AT_p^q(\a)/\Upsilon_k} \geq C_4\|f\|_{AT_p^q(\a)/\Upsilon_k}.
$$
Using Lemma \ref{Tfn}, we have
\begin{align*}
	&\int_{\mathbb{T}}\left(\int_{\Gamma_\tau(\eta)}|f^{(k)}(z)|^p|g^{(n-k)}(z)|^p(1-|z|^2)^{np+\a}d A(z)\right)^{\frac{q}{p}}|d \eta| \\
	\geq& C_4 \int_{\mathbb{T}}\left(\int_{\Gamma_\tau(\eta)}|f^{(k)}(z)|^p(1-|z|^2)^{kp+\a}d A(z)\right)^{\frac{q}{p}}|d \eta|.
\end{align*}
Using \cite[Theorem 5.4]{z1}, we obtain
$$
\begin{aligned}
	&\int_{\mathbb{T}}\left(\int_{\Gamma_\tau(\eta)}|f^{(k)}(z)|^p|g^{(n-k)}(z)|^p(1-|z|^2)^{np+\a}d A(z)\right)^{\frac{q}{p}}|d \eta| \\
	\lesssim& \int_{\mathbb{T}}\left(\int_{\Gamma_\tau(\eta) \cap G_c^{n,k}}|f^{(k)}(z)|^p|g^{(n-k)}(z)|^p(1-|z|^2)^{np+\a}d A(z)\right)^{\frac{q}{p}}|d \eta| \\
	&+\int_{\mathbb{T}}\left(\int_{\Gamma_\tau(\eta) \backslash G_c^{n,k}}|f^{(k)}(z)|^p|g^{(n-k)}(z)|^p(1-|z|^2)^{np+\a}d A(z)\right)^{\frac{q}{p}}|d \eta| \\
	\leq&\|g\|_{\mathcal{B}}^q \int_{\mathbb{T}}\left(\int_{\Gamma_\tau(\eta) \cap G_c^{n,k}}|f^{(k)}(z)|^p(1-|z|^2)^{kp+\a}d A(z)\right)^{\frac{q}{p}}|d \eta|  \\
	&+c^q \int_{\mathbb{T}}\left(\int_{\Gamma_\tau(\eta)}|f^{(k)}(z)|^p(1-|z|^2)^{kp+\a}d A(z)\right)^{\frac{q}{p}}|d \eta|  .
\end{aligned}
$$
Hence, we get
\begin{align*}
	&\|g\|_{\mathcal{B}}^q \int_{\mathbb{T}}\left(\int_{\Gamma_\tau(\eta) \cap G_c^{n,k}}|f^{(k)}(z)|^p(1-|z|^2)^{kp+\a}d A(z)\right)^{\frac{q}{p}}|d \eta|  \\
	&+c^q \int_{\mathbb{T}}\left(\int_{\Gamma_\tau(\eta)}|f^{(k)}(z)|^p(1-|z|^2)^{kp+\a}d A(z)\right)^{\frac{q}{p}}|d \eta|\\
	\ge&C_4 \int_{\mathbb{T}}\left(\int_{\Gamma_\tau(\eta)}|f^{(k)}(z)|^p(1-|z|^2)^{kp+\a}d A(z)\right)^{\frac{q}{p}}|d \eta|.
\end{align*}
Therefore, choosing $c>0$ small enough, there exists $C_5>0$ such that 
\begin{align}\label{cls}
\int_{\mathbb{T}}\left(\int_{\Gamma_\tau(\eta) \cap G_c^{n,k}}|f^{(k)}(z)|^p(1-|z|^2)^{kp+\a}d A(z)\right)^{\frac{q}{p}}|d \eta|
\ge C_5 \|f\|_{AT_p^q(\a)/\Upsilon_k}.
\end{align}

 For $\lambda>\max \left\{1, \frac{p}{q}\right\}$ and $a\in \DD$. Consider   the function $f_a$ defined in Proposition \ref{3.5}. Applying Lemma \ref{2.3}, Proposition \ref{3.5} and \cite[(3.4)]{ag2}, there exists $0<\rho<1$ such that
\begin{align} \label{cls2}
 \left\|f_a \chi_{G_c^{n,k}\cap \Delta(a, \rho)}\right\|_{AT_p^q(\a)/\Upsilon_k}^p 
\asymp&\left(\int_{\mathbb{T}}\left(\int_{\Gamma_{\tau}(\eta)} \chi_{G_c^{n,k}\cap \Delta(a, \rho)}(z)|f_a(z)|^p (1-|z|^2)^{\a}dA(z)\right)^{\frac{q}{p}}|d \eta|\right)^{\frac{p}{q}}\nonumber \\	
\asymp&
{(1-|a|^2)^{-2-\frac{p}{q}}}\left(\int_{\mathbb{T}}\left(\int_{\Gamma_{\tau}(\eta)} \chi_{G_c^{n,k}\cap \Delta(a, \rho)}(z) dA(z)\right)^{\frac{q}{p}}|d \eta|\right)^{\frac{p}{q}} \nonumber\\
\asymp& \frac{(1-|a|^2)^{-\frac{p}{q}}}{A(\Delta(a, \rho))}\left(\int_{\mathbb{T}}\left(\int_{G_c^{n,k}\cap \Delta(a, \rho)} \left( \frac{1-|z|^2}{|1-z \ol{\eta}|} \right)^\lambda dA(z)\right)^{\frac{q}{p}}|d \eta|\right)^{\frac{p}{q}}\nonumber\\
\lesssim  & \frac{(1-|a|^2)^{-\frac{p}{q}} }{A(\Delta(a, \rho))}    A(G_c^{n,k}\cap\Delta(a,\rho))(1-|a|^2)^{\frac{p}{q}}\nonumber\\
= & \frac{A(G_c^{n,k}\cap\Delta(a,\rho))}{A(\Delta(a,\rho))} ,
\end{align} for all $a \in \mathbb{D}$.
By Proposition \ref{3.5}, there are constants $C',C''>0$ such that $C''\le \|f_a\|_{AT_p^q(\a)}\le C'$. Let \( 0 < \varepsilon < \frac{ C_5C''}{2} \).
Using (\ref{cls}) and (\ref{cls2}), there exists \(0 < \rho < 1\) such that  
$$
\begin{aligned}
\left(\frac{A(G_c^{n,k}\cap \Delta(a, \rho))}{A(\Delta(a, \rho))}\right)^{\frac{1}{p}} & \geq\left\|f_a \chi_{G_c^{n,k}\cap \Delta(a, \rho)}\right\|_{AT_p^q(\a)/\Upsilon_k} \\
&\geq\left\|f_a \chi_{G_c^{n,k}}\right\|_{AT_p^q(\a)/\Upsilon_k}-\left\|f_a \chi_{\Delta^c(a, \rho)}\right\|_{AT_p^q(\a)/\Upsilon_k} \\
& \geq C_5\|f_a\|_{AT_p^q(\a)/\Upsilon_k}-\varepsilon \geq  C_5C''-\varepsilon \geq \frac{ C_5C''}{2}
\end{aligned}
$$
for all $a \in \mathbb{D}$.  Therefore, there exist constants \( \delta > 0 \) and \( 0 < \rho < 1 \) such that
\[
A(G_c^{n,k}\cap \Delta(a, \rho)) \geq \delta A(\Delta(a, \rho))
\]
for all \( a \in \mathbb{D} \).

$(ii)\Leftrightarrow(iii)$  This has been shown in \cite{l0}.

$(iii)\Rightarrow(i)$ Assume that $(iii)$ holds. Let $B$ denote the set described in Lemma \ref{b}. Let \( 0 < \varepsilon < 1 \), \( 0 < \lambda < \frac{1}{2^p} \) and $n\in\mathbb{N}\cup\{0\}$.
If $$a \in \mathbb{D} \backslash B=\left\{a \in \mathbb{D}:|f^{(n)}(a)|^p\ge\varepsilon^{1+\frac{2}{p}} B_\lambda f(a)\right\},$$  we have
$$
\frac{B_\lambda f(a)}{|f^{(n)}(a)|^p} \leq \frac{1}{\varepsilon^{1+\frac{2}{p}}} .
$$
If $0<\delta<1$, we can choose $\lambda<\varepsilon^{\frac{2}{\delta}\left(1+\frac{2}{p}\right)}$. By Lemma \ref{1.}, we obtain
$$
\frac{A\left(E_\lambda(a)\right)}{A\left(E(a,\rho)\right)}>\frac{(2 / \delta) \log \left(1 / \varepsilon^{1+\frac{2}{p}}\right)}{\log \left(1 / \varepsilon^{1+\frac{2}{p}}\right)+(2 / \delta) \log \left(1 / \varepsilon^{1+\frac{2}{p}}\right)}>1-\frac{\delta}{2}.
$$
Thus,
\begin{align}\label{c}
	A\left(E_\lambda(a)\right)>\left(1-\frac{\delta}{2}\right) A\left(E(a,\rho)\right).
\end{align}
Let  \( \lambda < \min \left\{ \frac{1}{2^p}, \varepsilon^{\frac{2}{\delta}(1+\frac{2}{p})} \right\} \). By  (\ref{c}), we have
$$
\begin{aligned}
A\left(G_c^{n,k}\cap E_\lambda(a)\right) & =A\left(G_c^{n,k}\cap E(a,\rho)\right)-A\left(G_c^{n,k}\cap\left(E(a,\rho) \backslash E_\lambda(a)\right)\right) \\
& \geq \delta A\left(E(a,\rho)\right)-A\left(E(a,\rho) \backslash E_\lambda(a)\right) \\
& =\delta A\left(E(a,\rho)\right)-A\left(E(a,\rho)\right)+A\left(E_\lambda(a)\right) \\
& \geq \delta A\left(E(a,\rho)\right)-A\left(E(a,\rho)\right)+\left(1-\frac{\delta}{2}\right) A\left(E(a,\rho)\right) \\
& =\frac{\delta}{2} A\left(E(a,\rho)\right),
\end{aligned}
$$
for all $a \in \mathbb{D} \backslash B$.

Let $f\in AT_p^q(\a)$, $\eta\in\TT$ and $a \in \Gamma_\tau(\eta) \backslash B$. There exists $\tau'\in (\tau,1)$ such that  $$E_\lambda(a) \subset E(a,\rho) \subset \Gamma_{\tau'}(\eta).$$
So, 
\begin{align*}
&\frac{1}{A\left(E(a,\rho)\right)} \int_{G_c^{n,k}\cap \Gamma_{\tau'}(\eta)} \chi_{E(a,\rho)}(z)|f^{(n)}(z)|^p(1-|z|^2)^{np+\a}d A(z) \\
\geq&  \frac{\delta}{2A\left(G_c^{n,k}\cap E_\lambda(a)\right)} \int_{G_c^{n,k}\cap \Gamma_{\tau'}(\eta)} \chi_{E(a,\rho)}(z)|f^{(n)}(z)|^p(1-|z|^2)^{np+\a}d A(z) \\
\geq& \frac{\delta}{2A\left(G_c^{n,k}\cap E_\lambda(a)\right)} \int_{G_c^{n,k}\cap E_\lambda(a)} \chi_{E(a,\rho)}(z)|f^{(n)}(z)|^p(1-|z|^2)^{np+\a}d A(z) \\
\gtrsim& \frac{\delta \lambda}{2} |f^{(n)}(a)|^p(1-|a|^2)^{np+\a} .
\end{align*}
Integrating over the set $\Gamma_\tau(\eta) \backslash B$ and using Fubini's theorem on the left-hand side, we get
\begin{align*}
&\int_{G_c^{n,k}\cap \Gamma_{\tau'}(\eta)}|f^{(n)}(z)|^p(1-|z|^2)^{np+\a} \int_{\Gamma_\tau(\eta) \backslash B} \frac{\chi_{E(a, \rho)}(z)}{A\left(E(a,\rho)\right)} dA(a) dA(z) \\
\geq&  \frac{\delta \lambda}{2} \int_{\Gamma_\tau(\eta) \backslash B} |f^{(n)}(a)|^p(1-|a|^2)^{np+\a}dA(a).
\end{align*}
When $\frac{2 \rho}{1+\rho^2} \leq r<1$, using $\chi_{E(a,\rho)}(z) \leq \chi_{\Delta(a, r)}(z)=\chi_{\Delta(z,r)}(a)$ again
  we have
\begin{align*}
&\int_{G_c^{n,k}\cap \Gamma_{\tau'}(\eta)}|f^{(n)}(z)|^p(1-|z|^2)^{np+\a} \int_{\Gamma_\tau(\eta) \backslash B} \frac{\chi_{\Delta(z, r)}(a)}{A\left(E(a,\rho)\right)} dA(a) dA(z) \\
\geq&  \frac{\delta \lambda}{2} \int_{\Gamma_\tau(\eta) \backslash B} |f^{(n)}(a)|^p(1-|a|^2)^{np+\a}dA(a).
\end{align*}
Similar to the estimate in Lemma \ref{a}, we can obtain that there exists a positive constant \( C_6 \), which depends only on \( \rho \), such that
$$
\int_{\Gamma_\tau(\eta) \backslash B} \frac{\chi_{\Delta(z, r)}(a)}{A(E(a, \rho))} d A(a) \leq C_6,  \quad z\in\D.
$$
Thus, 
\begin{align*}
C_6 \int_{G_c^{n,k}\cap \Gamma_{\tau'}(\eta)}|f^{(n)}(z)|^p(1-|z|^2)^{np+\a} dA(z) 
\geq  \frac{\delta \lambda}{2} \int_{\Gamma_\tau(\eta) \backslash B} |f^{(n)}(a)|^p(1-|a|^2)^{np+\a}dA(a).
\end{align*}
Using Lemma \ref{b}, we obtain
\begin{align*}
&\int_{G_c^{n,k}\cap \Gamma_{\tau'}(\eta)}|f^{(n)}(z)|^p(1-|z|^2)^{np+\a}dA(z)\\
\ge& \frac{\delta \lambda}{2C_6} \int_{\Gamma_\tau(\eta) \backslash B} |f^{(n)}(a)|^p(1-|a|^2)^{np+\a}dA(a) \\
=&\frac{\delta \lambda}{2C_6} \int_{\Gamma_\tau(\eta) } |f^{(n)}(a)|^p(1-|a|^2)^{np+\a}dA(a)-\frac{\delta \lambda}{2C_6} \int_{\Gamma_\tau(\eta) \cap B} |f^{(n)}(a)|^p(1-|a|^2)^{np+\a}dA(a) \\
\ge&\frac{\delta \lambda}{2C_6} \int_{\Gamma_\tau(\eta) } |f^{(n)}(a)|^p(1-|a|^2)^{np+\a}dA(a)-\frac{C_2\varepsilon\delta \lambda}{2C_6} \int_{\Gamma_{\tau'}(\eta)} |f^{(n)}(a)|^p(1-|a|^2)^{np+\a}dA(a).
\end{align*}
Thus,
\begin{align*}
&\int_{G_c^{n,k}\cap \Gamma_{\tau'}(\eta)}|f^{(n)}(z)|^p(1-|z|^2)^{np+\a}dA(z)+\frac{C_2\varepsilon\delta \lambda}{2C_6} \int_{\Gamma_{\tau'}(\eta)} |f^{(n)}(a)|^p(1-|a|^2)^{np+\a}dA(a)\\
\ge&\frac{\delta \lambda}{2C_6} \int_{\Gamma_\tau(\eta) } |f^{(n)}(a)|^p(1-|a|^2)^{np+\a}dA(a).
\end{align*}
Hence,
\begin{align*}
&\left( \int_{G_c^{n,k}\cap \Gamma_{\tau'}(\eta)}|f^{(n)}(z)|^p(1-|z|^2)^{np+\a}dA(z) \right)^\frac{1}{p}\\
&+\left(\frac{C_2\varepsilon\delta \lambda}{2C_6} \right)^\frac{1}{p}  \left(\int_{\Gamma_{\tau'}(\eta)} |f^{(n)}(a)|^p(1-|a|^2)^{np+\a}dA(a)\right)^\frac{1}{p}\\
\ge&\left(\frac{\delta \lambda}{2C_6}\right)^\frac{1}{p} \left(\int_{\Gamma_\tau(\eta) } |f^{(n)}(a)|^p(1-|a|^2)^{np+\a}dA(a)\right)^\frac{1}{p}.
\end{align*}
Given that $q\ge1$, applying Minkowski's inequality, we get
\begin{equation}\label{11}
\begin{aligned}
&\left(\int_{\mathbb{T}}\left( \int_{G_c^{n,k}\cap \Gamma_{\tau'}(\eta)}|f^{(n)}(z)|^p(1-|z|^2)^{np+\a}dA(z) \right)^\frac{q}{p} |d\eta|\right)^{\frac{1}{q}}\\
&+\left(\frac{C_2\varepsilon\delta \lambda}{2C_6} \right)^\frac{1}{p}  \left(\int_{\mathbb{T}}\left(\int_{\Gamma_{\tau'}(\eta)} |f^{(n)}(a)|^p(1-|a|^2)^{np+\a}dA(a)\right)^\frac{q}{p}|d\eta|\right)^{\frac{1}{q}}\\
\ge&\left(\frac{\delta \lambda}{2C_6}\right)^\frac{1}{p} \left(\int_{\mathbb{T}}\left(\int_{\Gamma_\tau(\eta) } |f^{(n)}(a)|^p(1-|a|^2)^{np+\a}dA(a)\right)^\frac{q}{p}|d\eta|\right)^{\frac{1}{q}}.
\end{aligned}
\end{equation}
Since \( 0 < \tau < {\tau'} < 1 \),  there exists a positive constant \( C_7 \) that depends only on \( \rho \) such that
\begin{equation}\label{12}
\begin{aligned}
&\left(\int_{\mathbb{T}}\left(\int_{\Gamma_{\tau'}(\eta) } |f^{(n)}(a)|^p(1-|a|^2)^{np+\a}dA(a)\right)^\frac{q}{p}|d\eta|\right)^{\frac{1}{q}}\\
\le&C_7\left(\int_{\mathbb{T}}\left(\int_{\Gamma_\tau(\eta) } |f^{(n)}(a)|^p(1-|a|^2)^{np+\a}dA(a)\right)^\frac{q}{p}|d\eta|\right)^{\frac{1}{q}}.
\end{aligned}
\end{equation}
Choosing \( \varepsilon \) sufficiently small such that \(  \left(1 - C_2^{\frac{1}{p}} C_7 \varepsilon^{\frac{1}{p}} \right)> 0 \). Combining $(\ref{11})$ and $(\ref{12})$, it follows that
\begin{align*}
&\left(\int_{\mathbb{T}}\left( \int_{G_c^{n,k}\cap \Gamma_{\tau'}(\eta)}|f^{(n)}(z)|^p(1-|z|^2)^{np+\a}dA(z) \right)^\frac{q}{p} |d\eta|\right)^{\frac{1}{q}}\\
\ge& \left(1-C_2^{\frac{1}{p}}C_7\varepsilon^\frac{1}{p}\right)\left( \frac{\delta\lambda}{2C_6}\right)^{\frac{1}{p}} \left(\int_{\mathbb{T}}\left(\int_{\Gamma_\tau(\eta)} |f^{(n)}(a)|^p(1-|a|^2)^{np+\a}dA(a)\right)^\frac{q}{p}|d\eta|\right)^{\frac{1}{q}}\\
\ge&\left(1-C_2^{\frac{1}{p}}C_7\varepsilon^\frac{1}{p}\right)\left( \frac{\delta\lambda}{2C_6}\right)^{\frac{1}{p}} \|f\|_{AT_p^q(\a)/\Upsilon_k}.
\end{align*}
Since 
$G_c^{n,k}=\left\{z \in \mathbb{D}:|g^{(n-k)}(z)|(1-|z|^2)^{n-k}>c\right\},$ 
 using the fact that $ T_g^{n,k}f(0) = (T_g^{n,k}f)'(0)=\cdots=(T_g^{n,k}f)^{(n-1)}(0)= 0 $
and Lemma \ref{Tfn}, we obtain
$$
\begin{aligned}
 \| T_g^{n,k}f \|_{AT_p^q(\a)/\Upsilon_k} 
=&\left(\int_{\mathbb{T}}\left(\int_{\Gamma_{\tau'}(\eta)}|f^{(k)}(z)|^p|g^{(n-k)}(z)|^p(1-|z|^2)^{np+\a} dA(z)\right)^{\frac{q}{p}}|d\eta|\right)^{\frac{1}{q}} \\
\geq&\left(\int_{\mathbb{T}}\left(\int_{\Gamma_{\tau'}(\eta) \cap G_c^{n,k}}|f^{(k)}(z)|^p|g^{(n-k)}(z)|^p(1-|z|^2)^{np+\a} dA(z)\right)^{\frac{q}{p}}|d\eta|\right)^{\frac{1}{q}} \\
\geq& c\left(\int_{\mathbb{T}}\left(\int_{\Gamma_{\tau'}(\eta) \cap G_c^{n,k}}|f^{(k)}(z)|^p(1-|z|^2)^{kp+\a} dA(z)\right)^{\frac{q}{p}}|d\eta|\right)^{\frac{1}{q}}\\
\geq& c\left(1-C_2^{\frac{1}{p}}C_7\varepsilon^\frac{1}{p}\right)\left( \frac{\delta\lambda}{2C_6}\right)^{\frac{1}{p}} \| f \|_{AT_p^q(\alpha)/\Upsilon_k}.
\end{aligned} 
$$ 
The proof is complete.
\end{proof}

\noindent \textbf{Acknowledgments.} 
 The   authors are supported  by GuangDong Basic and Applied Basic Research Foundation (No. 2023A1515010614). \\

\noindent\textbf{Data Availability.}  Data sharing is not applicable for this article as no datasets were generated or analyzed during the current study. \\

 \noindent \textbf{Conflict of interest.}   The authors declare no competing interests.

\end{document}